\documentclass[12pt]{article}
\usepackage{palatino}
\usepackage{aliascnt}
\usepackage{amsmath}
\usepackage{amsrefs}
\usepackage{amssymb}
\usepackage{amsthm}
\usepackage{algorithm}
\usepackage{algpseudocode}
\usepackage{tabularx}

\usepackage{bm}
\usepackage{bbm}

\usepackage{calrsfs}
\usepackage{color}

\usepackage[sans]{dsfont}

\usepackage{enumerate}
\usepackage{epsfig}
\usepackage{extarrows}
\usepackage[mathscr]{euscript}
\usepackage[symbol]{footmisc}
\usepackage{framed}
\usepackage{fullpage}
\usepackage[margin=1in]{geometry}
\usepackage{graphicx}
\usepackage{import}
\usepackage{multirow}

\usepackage{paralist}
\usepackage{psfrag}
\usepackage{amsmath}
\usepackage{subfig}
\usepackage[]{times}
\usepackage{tikz}
\usepackage{transparent}
\usepackage[normalem]{ulem}
\usepackage{url}
\usepackage{verbatim}
\usepackage{wrapfig}
\usepackage{xcolor}
\usepackage{xifthen}
\definecolor{Green}{rgb}{0.0,0.45,0.0}
\usepackage[colorlinks,linkcolor=blue,citecolor=red,bookmarksopen=false]{hyperref}

\def\NewTheorem#1#2{%
	\newaliascnt{#1}{thmm}
	\newtheorem{#1}[#1]{#2}
	\aliascntresetthe{#1}
	\expandafter\def\csname #1autorefname\endcsname{#2}
}

\numberwithin{equation}{section}

\NewTheorem{lem}{Lemma}
\NewTheorem{thm}{Theorem}
\NewTheorem{cor}{Corollary}
\NewTheorem{prop}{Proposition}
\theoremstyle{definition}
\NewTheorem{defin}{Definition}
\theoremstyle{remark}
\NewTheorem{remark}{Remark}
\theoremstyle{definition}
\NewTheorem{eg}{Example}

\title{A Note on 
Central Limit Theorems for Additive Functionals of Ergodic Markov Processes}
\begin{document}

\author{Edward C.\ Waymire\thanks{Department of Mathematics,  Oregon State University, Corvallis, OR, 97331. {waymire@oregonstate.edu}}
}

\maketitle

\begin{abstract}
This note provides a simple proof of the highly cited
Kipnis-Varadhan central limit theorem for 
additive functionals of 
a time-reversible 
continuous parameter Markov process 
as a corollary to Bhattacharya's 
general central limit theorem for additive
functionals of ergodic Markov processes.  This makes
the Kipnis-Varadhan and Bhattacharya
 central limit theorems equivalent for the case of
time-reversible Markov processes. 
 This is revealed in a fascinating \lq\lq double-limit problem\rq\rq
 that is resolved by a simple use of 
Bhattacharya's range condition on the infinitesimal generator. 
A new version of Bhattacharya's law of the iterated logarithm is also included for additive functionals of time-reversible Markov processes as an application of the methods of the present paper.

\end{abstract}

\section{Introduction and Preliminaries}
\label{introsec}
This note provides a simple proof of the Kipnis-Varadhan central limit theorem \cite{kipvar} for additive functionals $\int_0^tf(X(s))ds$ 
 of a time-reversible continuous parameter Markov process
 $X$ with 
 infinitesimal generator\footnote{The `hat' notation
is adopted from \cite{RBfclt} to
 signify the infinitesimal generator of the semigroup extension
from the Banach space   of bounded, measurable functions
to $L^2(S,\mathcal{S},\pi)$.  It will continue to be used
in reference to operators defined on Hilbert spaces.
} $\hat{A}$, for $f$ belonging
  to the domain of $(-\hat{A})^{-\frac{1}{2}}$. 
  Given the rather vast literature on this theorem for self-adjoint
$\hat{A}$, it may come
as a surprise that  this renders Bhattacharya's central limit theorem and
the Kipnis-Varadhan central limit theorem 
to be equivalent for the case of time-reversible Markov processes. 
To this end, a fascinating\footnote{The highlighted importance of double-limit problems to mathematics
is often attributed to G.H. Hardy (1928) Appendix III to {\it A Course in Pure Mathematics}, Cambridge Mathematical Library, Cambridge Press.} \lq\lq double-limit problem\rq\rq is
identified and resolved by a simple  use of 
Bhattacharya's general range condition
$f\in\mathcal{R}_{\hat{A}}$ given in \cite{RBfclt}; see 
Theorem \ref{rabifcltthm} below.  

A notable take-away of the method presented here is that in the case of time-reversible Markov
processes there is no loss of generality in restricting the proof
to integrands
belonging to the range of the generator.  This is also
exemplified with a new version of Bhattacharya's law of the iterated logarithm for additive functionals of time-reversible 
Markov processes to conclude this article.
For ease of reference both central limit theorems 
from \cites{RBfclt, kipvar} are stated
 in the next section.

 Aided by the special structure furnished by the added
 self-adjointness condition, the original proof provided in \cite{kipvar} is quite deep in its
 use of spectral theory for {\it unbounded} self-adjoint operators
 and factorization $-\hat{A}^{-1} = (-\hat{A})^{-\frac{1}{2}}(-\hat{A})^{-\frac{1}{2}}$
 in an analysis of a double-limit in which 
 a resolvent transform 
 parameter $\lambda\downarrow0$ is followed by the 
 scaling parameter
 $n\to\infty$ described below.  
 In a subsequent 
  paper \cite{varadhan95},  a more 
  \lq\lq functional analytic\rq\rq 
  proof is provided in the self-adjoint case
 for the double-limit approach of
  \cite{kipvar}, but without specific use of spectral theory in
  key estimates.  This approach
 is also developed in lengthier
detail in \cites{liggett99,landim,sunder}
where a novel Hilbert space convexity property 
is proven and exploited in the double-limit of \cite{kipvar}.
 Apart from such non-trivial technicalities,
 the essence of these proofs can be described as follows.
 
 To set the framework as in \cite{kipvar},
 consider the following 
identity 
 for the resolvent operators, made obvious by their definition $R_\lambda f = (\lambda-\hat{A})^{-1}f$,
$\lambda >0$. Namely,
\begin{equation}\label{frep}
f = \lambda R_\lambda f - 
\hat{A}R_\lambda f, \quad \lambda >0.
\end{equation} 
Then, for each $n\ge 1$,
\begin{equation}\label{intrep}
\frac{1}{\sqrt{n}}\int_0^{nt}f(X(s))ds 
= \frac{1}{\sqrt{n}}\int_0^{nt}\lambda R_\lambda f(X(s))ds
+ \frac{1}{\sqrt{n}}\int_0^{nt}(-\hat{A})R_\lambda f(X(s))ds,
\  \lambda>0, t\ge 0.
\end{equation}
It is convenient to write the three 
sequences of processes 
corresponding to the terms 
appearing in (\ref{intrep}) from left to right, as $\{I_n(f,t):t\ge0\}$, 
$\{\Lambda_n(f,\lambda,t):t\ge 0\}$, and $\{I_n(-\hat{A}R_\lambda f,t): t\ge 0\}$, respectively. In this notation, 
(\ref{intrep}) may be expressed as 
\begin{equation}
I_n(f,t)=\Lambda_n(f,\lambda,t)+I_n(-\hat{A}R_{\lambda}f,t), \quad
t\ge0.
\end{equation}
Also denote the (unique) invariant probability by $\pi$ and
the corresponding innerproduct on $L^2(S,\pi)$ by $\langle\cdot,\cdot\rangle_\pi$, and define $1^\perp = \{f\in L^2(S,\pi):\langle f,1\rangle_\pi=0\}$.
To
 control $\Lambda_n(f,\lambda,\cdot)$ as a function of $n$ and $\lambda>0$, observe that
using (\ref{frep}) one has
\begin{equation}
\langle f,R_\lambda f\rangle_\pi =
\langle \lambda R_\lambda f,R_\lambda f\rangle_\pi
+ \langle-\hat{A}R_\lambda f,R_\lambda f\rangle_\pi
=\lambda||R_\lambda f||_\pi^2 + ||(-\hat{A})^{\frac{1}{2}}R_\lambda f||^2_\pi
\end{equation}
Thus, 
\begin{equation*}
\lambda||R_\lambda f||_\pi^2 + ||(-\hat{A})^{\frac{1}{2}}R_\lambda f||^2_\pi 
=|\langle f,R_\lambda f\rangle_\pi|
\le ||(-\hat{A})^{-\frac{1}{2}}f||_\pi ||(-\hat{A})^{\frac{1}{2}}R_\lambda f||_\pi,
\end{equation*}
and therefore
\begin{equation}\label{lambdaRlambdabound}
\lambda^{\frac{1}{2}}||R_\lambda f||_\pi
 \le\frac{1}{2}||(-\hat{A})^{-\frac{1}{2}}f||_\pi, \quad \lambda >0.
\end{equation}
Now, by Jensen's inequality applied to the suitably rescaled integral
$\int_0^{nt}$,
\begin{equation}\label{jensenbnd}
\mathbb{E}(\max_{0\le t\le T}\frac{1}{\sqrt{n}}\int_0^{nt}\lambda R_\lambda f(X(s))ds)^2
\le nT^2||\lambda R_\lambda f||^2_\pi
\le \frac{1}{4}nT^2\lambda||(-\hat{A})^{-\frac{1}{2}} f||^2_\pi.
\end{equation}  
So, along any sequence of values decreasing to zero of
 \begin{equation}
0< \lambda_n=o(\frac{1}{n}), 
 \end{equation}
 one has that
\begin{equation}\label{o(1)bound}
\max_{0\le t\le T}\Lambda_n(f,\lambda_n,t) 
=o(1)\quad \text{in probability as}\  n\to\infty. 
\end{equation}
So it follows that
$I_{n}(f,\cdot)$ and $I_{n}(f,\cdot)-\Lambda_{n}(f,\lambda_{n},\cdot) = I_n(-\hat{A}R_{\lambda_n}f,\cdot)$ have the same distribution
in the limit as $n\to\infty$.

In the aforementioned approaches \cites{kipvar, varadhan95, liggett99, landim, sunder}, the authors first add and subtract
extra terms\footnote{Notably, in the case $f=\hat{A}g\in\mathcal{R}_{\hat{A}}$,
up to a closed formula for the variance, 
the Gaussian limit of $I_{n}(f,\cdot)$ as $n\to\infty$
 proven in \cite{RBfclt}
 is made reasonably straightforward from the martingale central
 limit theorem and the fact that 
$g(X_t)-\int_0^tf(X_s)ds, t\ge0$, is naturally a martingale.}
(\ref{intrep}) in order to explicitly express
$I_n(f,\cdot)$ in terms of 
Dynkin martingale.  
The limit $\lambda\downarrow 0$ is then shown to exist,
and attention is turned to
analyzing that result in a second limit as
 $n\to\infty$.  The second stage is where the technicalities emerge,
and naturally beg the question regarding the nature and extent
of a reversal in the order of limits ?
Namely, on purely formal grounds, 
 the problem may be viewed as justifying the 
 following formal interchange of limits and invoking \cite{RBfclt}:
 \begin{eqnarray}\label{limitorder}
 \lim_{n\to\infty}{I}_n(f,t) 
 &=& \lim_n\lim_{\lambda\downarrow0}\{\Lambda_{n}(f,t) +
 I_{n}(-AR_\lambda f,t)\}\nonumber\\
 &=&\lim_{n} \lim_{\lambda\downarrow0}I_{n}(-AR_\lambda f,t)\nonumber\\
 &=&\lim_{\lambda\downarrow0} \lim_nI_{n}(-AR_\lambda f,t)\nonumber\\
 &=&\lim_{\lambda\downarrow0}N(0,\sigma^2(-\hat{A}R_\lambda f))
 = N(0,\sigma^2(f)) ?
 \end{eqnarray}
 
 As a final preliminary before turning to the proof in the next section, 
it may be helpful to have a sketch of a well-known construction
 of the square root of a positive self-adjoint operator
$L$ on a
Hilbert space $H$,
  i.e.,
 $\langle Lf,f\rangle\ge 0, \forall f\in H$. 
The resolvent operators $R_\lambda=(\lambda + L)^{-1},
\lambda>0$,
are  bounded, self-adjoint
positive operators with square roots
 $S_\lambda = R_\lambda^{\frac{1}{2}}, \lambda >0$.
Letting $U_\lambda=S_\lambda^{-1}$, one
 may show that
 $Uf :=\lim_{\lambda\downarrow0}U_\lambda f$,
 exists for all $f$ in the domain $\mathcal{D}_{L}$, and defines
  a positive, self-adjoint operator such that $U^2=L$,
  with $\mathcal{D}_U=\mathcal{D}_{U_1}= 
  \mathcal{D}_{U_\lambda}, \lambda >0$; see \cite{bernau}.
In the proof
of Corollary \ref{equiv} one may consider $L=(-\hat{A})^{-1}$  on 
$\mathcal{D}_{(-\hat{A})^{-1}}
= \mathcal{R}_{\hat{A}}\subset H=1^\perp$ and hence,
$U=(-\hat{A})^{-\frac{1}{2}}$ with
$\mathcal{R}_{\hat{A}}=\mathcal{D}_{(-\hat{A})^{-1}}\subset\mathcal{D}_{(-\hat{A})^{-\frac{1}{2}}}$.

\section{Proof of the Kipnis-Varadhan Theorem and its
Equivalence to Bhattacharya's Theorem}
For ease of reference both theorems are stated to start.
\begin{thm}[Bhattacharya \cite{RBfclt}]
\label{rabifcltthm}
Suppose that $X = \{X(t):t\ge0\}$ is a progressively measurable
ergodic continuous parameter
Markov process on a measurable state space $(S,{\mathcal S})$
starting from a unique invariant probability $\pi$, and defined
on a complete probability space $(\Omega,{\mathcal F},P_\pi)$.   Then for centered
$f\in 1^\perp\subset L^2(S,\mathcal{S},\pi)$, i.e., 
$\int_Sfd\pi=0$,
 belonging to the range ${\mathcal R}_{\hat{A}}$
of a densely defined, closed  infinitesimal generator 
$(\hat{A},{\mathcal D}_{\hat{A}})
\subset L^2(S,\mathcal{S},\pi)$,  the sequence 
$n^{-\frac{1}{2}}\int_0^{nt}f(X(s))ds: t\ge 0\}, n\ge 1$, converges weakly in $C[0,\infty)$ to Brownian
motion starting at $0$ with zero drift and 
diffusion coefficient 
\begin{equation}\label{diffcoeff}
\sigma^2(f) = 2\langle -\hat{A}^{-1}f,f\rangle_\pi 
\end{equation}
\end{thm}
\begin{remark}  Although not required for the 
proof here, there are a few simple
facts from standard semigroup theory that seem 
worth noting.
For strongly 
continuous semigroups, 
the assumption of a {\it closed} infinitesimal
generator follows from the Hille-Yosida theorem.
In particular,  $\hat{A}$ is a densely defined {\it closed}
operator on $L^2(S,\mathcal{S},\pi)$. 
It also obvious that for self-adjoint $\hat{A}$
and non-trivial $f\in\mathcal{R}_{\hat{A}}$, one has
$\sigma^2(f)= 2||(-A)^{\frac{1}{2}}g||_\pi^2>0$, $f=\hat{A}g$.
 Moreover, a densely defined self-adjoint operator is closed
 since the adjoints of densely defined 
 operators are always closed. 
\end{remark}
 \begin{thm}[Kipnis-Varadhan \cite{kipvar}]
\label{kipvarthm}
In the framework of Theorem \ref{rabifcltthm}
assume that $\hat{A}$ is a densely defined, self-adjoint infinitesimal
generator.
 If $f\in\mathcal{D}_{(-\hat{A})^{-\frac{1}{2}}}$ is centered, then
the sequence 
$n^{-\frac{1}{2}}\int_0^{nt}f(X(s))ds: t\ge 0\}, n\ge 1$, converges weakly in $C[0,\infty)$ to Brownian
motion starting at $0$ with zero drift and 
diffusion coefficient  
  $0<\sigma^2(f)=2\langle(-\hat{A})^{-\frac{1}{2}}f,(-\hat{A})^{-\frac{1}{2}}f\rangle_\pi < \infty$.
 \end{thm}
\begin{proof}
Assume that $f\in\mathcal{D}_{(-\hat{A})^{-\frac{1}{2}}}\cap1^\perp$.
Let $\lambda >0$.  Consider the 
identity (\ref{frep}) for the resolvent operators $R_\lambda f$,
$\lambda >0$, and the corresponding representation
(\ref{intrep}) for $I_n(f,\cdot)$. 
Obviously,  
 $\hat{A}R_{\lambda_\ell} f\in\mathcal{R}_{\hat{A}}$ satisfies Bhattacharya's
 range condition in \cite{RBfclt} with $g=R_{\lambda_\ell} f$.
  Thus, the  dispersion rate $\sigma^2_{\lambda_\ell}$
   may be computed
  in the limit as $n\to\infty$,
 as $2\langle-\hat{A}R_{\lambda_\ell}f,
R_{\lambda_\ell}f\rangle_\pi$. Now,
using positive
 operator monotonicity of
$||(-\hat{A})^{\frac{1}{2}}R_\lambda f||_\pi = ||(\lambda(-\hat{A})^{-\frac{1}{2}}+(-\hat{A})^{\frac{1}{2}})^{-1}f||_\pi$,
for 
  $f\in\mathcal{D}_{(-\hat{A})^{-\frac{1}{2}}}$, one has in the limit $\lambda_\ell\downarrow0$,
\begin{eqnarray}\label{varlim}
&&\sigma^2(-\hat{A}R_{\lambda_\ell}f)\nonumber\\
&=&2\langle-\hat{A}R_{\lambda_\ell}f,
R_{\lambda_\ell}f\rangle_\pi
=2\langle(-\hat{A})^{\frac{1}{2}}R_{\lambda_n}f,
(-\hat{A})^{\frac{1}{2}}R_{\lambda_\ell}f\rangle_\pi\nonumber\\
&=&2\langle(-\hat{A})^{\frac{1}{2}}(\lambda_\ell-\hat{A})^{-1}f,(-\hat{A})^{\frac{1}{2}}(\lambda_\ell-\hat{A})^{-1}f\rangle_\pi\nonumber\\
&=&2\langle(\lambda_\ell(-\hat{A})^{-\frac{1}{2}}+(-\hat{A})^{\frac{1}{2}})^{-1}f,(\lambda_\ell(-\hat{A})^{-\frac{1}{2}}+(-\hat{A})^{\frac{1}{2}})^{-1}f\rangle_\pi
\uparrow 2\langle(-\hat{A})^{-\frac{1}{2}}f,(-\hat{A})^{-\frac{1}{2}}f\rangle_\pi\nonumber\\
&&
\end{eqnarray}
Now, let us consider
the sequence
$\{I_{n}(-AR_{\lambda_{n}}f,\cdot)\}_n$.
Using (i) the well-known
resolvent identity\footnote{For example, see
 (\cite{BhatWay3}, p. 25).} 
 $R_\lambda-R_\mu=(\mu-\lambda)R_\lambda R_\mu$,  
  (ii) a decreasing sequence $\{\lambda_{n}=o(\frac{1}{n})\}_{n=1}^\infty$ 
and (iii)  the bound (\ref{lambdaRlambdabound}), one has 
\begin{eqnarray}\label{cauchyprop}
&&\mathbb{E}\max_{0\le t\le T}|I_{n}(-AR_{\lambda_{n}}f,t)-I_{n}(-AR_{\lambda_{\ell}}f,t)|\nonumber\\
&\le &\frac{1}{\sqrt{n}}\int_0^{nT}\mathbb{E}|\hat{A}R_{\lambda_{n}}f(X(s))
-\hat{A}R_{\lambda_{\ell}}f(X(s))|ds\nonumber\\
&\le& \sqrt{n}T(|\lambda_{n}-\lambda_{\ell}|)||\hat{A}R_{\lambda_{n}}R_{\lambda_{\ell}}f||_\pi
= 
\begin{cases}\sqrt{n}T(\lambda_{n}-\lambda_{\ell})||\hat{A}R_{\lambda_{\ell}}R_{\lambda_{n}}f||_\pi \ &\text{if}\ \lambda_n\ge\lambda_\ell\\
\sqrt{n}T(\lambda_{\ell}-\lambda_{n})||\hat{A}R_{\lambda_{n}}R_{\lambda_{\ell}}f||_\pi \ &\text{if}\ \lambda_\ell\ge\lambda_n
\end{cases}\nonumber\\ 
&\le&\begin{cases}\sqrt{n}T(\lambda_{n}-\lambda_{\ell})2||R_{\lambda_{n}}f||_\pi \ &\text{if}\ \lambda_n\ge\lambda_\ell\\
\sqrt{n}T(\lambda_{\ell}-\lambda_{n})2||R_{\lambda_{\ell}}f||_\pi \ &\text{if}\ \lambda_\ell\ge\lambda_n
\end{cases}
\le\begin{cases}\sqrt{n}T\sqrt{\lambda_{n}}||(-\hat{A})^{-\frac{1}{2}}f||_\pi \ &\text{if}\ \lambda_n\ge\lambda_\ell\\
\sqrt{n}T\sqrt{\lambda_{\ell}}||(-\hat{A})^{-\frac{1}{2}}f||_\pi
 \ &\text{if}\ \lambda_\ell\ge\lambda_n
\end{cases}\nonumber\\ 
&=& o(1)\ \text{for}\  \lambda_n,\lambda_\ell = o(\frac{1}{n})\end{eqnarray}
e.g.,   one may take ${\lambda_n} = n^{-2}$ and $\ell$ an integer satisfying $\ell > n^{1-\delta}$  for some  $0<\delta<1/2$.
The extraneous parameter $\ell$ serves as a tuning parameter
for fixing small, positive values of $\lambda_\ell$. 

These estimates essentially complete the proof
 that Theorem \ref{rabifcltthm} implies Theorem \ref{kipvarthm}. 
To see why requires a bit of extra notation. Let $\rho$ denote the Prohorov metric for weak
convergence on $C[0,\infty)$. Fix an arbitrary $T>0$.
 Let $Q_n$ denote the distribution of 
 $\{I_n(f,t): 0\le t\le T\}$, and $Q_{n,\ell}$  the distribution
of the stochastic process $\{I_n(-\hat{A}R_{\lambda_\ell}f,t): 0\le t\le T\}$.
Also, let $Q_{\infty,\ell}$ be the distribution of Brownian motion
with dispersion coefficient $2\langle-\hat{A}R_{\lambda_{\ell}} f,R_{\lambda_\ell} f\rangle_\pi$,
and let $Q_{\infty,0}$ denote the distribution of
the Brownian motion with
zero drift and dispersion coefficient $\sigma^2(f)= 2||(-\hat{A})^{-\frac{1}{2}}f||_\pi^2$. 

Then, by the triangle inequality,
one has for arbitrary positive integer $n$ and real
parameter $\ell>0$, 
\begin{eqnarray}\label{diaglim}
\rho(Q_{n},Q_{\infty,0})
&\le&
\rho(Q_{n}, Q_{n,n})+
\rho(Q_{n,n},Q_{n,\ell})
  + \rho(Q_{n,\ell},Q_{\infty,\ell})
 +\rho(Q_{\infty,\ell},Q_{\infty,0})\nonumber\\
  &=& I + II+ III + IV,
 \end{eqnarray}
 and the left side is independent of the parameter $\ell>0$.
By (\ref{o(1)bound}), $I<\epsilon$ for all $n\ge N_\epsilon^{(1)}$ (independently of $\ell>0$).  By (\ref{varlim}), $IV<\epsilon$
for all $\ell>L_\epsilon^{(1)}$ (independently of $n$).
By Bhattacharya's theorem \cite{RBfclt}, $III<\epsilon$ for all
$n\ge N_\epsilon^{(2)}(\ell)$. Thus, for every
$\ell>L_\epsilon^{(1)}$ and every $n\ge N_\epsilon^{(1)}\vee N_\epsilon^{(2)}(\ell)$, one has
\begin{equation}
\rho(Q_{n},Q_{\infty,0})
\le 3\epsilon +
\rho(Q_{n,n},Q_{n,\ell}).
\end{equation}
In particular by (\ref{cauchyprop}), 
\begin{equation}
\limsup_{n\to\infty}\rho(Q_{n},Q_{\infty,0})
= \limsup_{n\to\infty,\ell=o(\frac{1}{n})}\rho(Q_{n},Q_{\infty,0})
\le 3\epsilon, 
\end{equation}
or, alternatively, make the bound $4\epsilon$ by choosing
$n,\ell$ sufficiently large. 
\end{proof}
\begin{cor}[Central Limit Theorem Equivalence]\label{equiv}
Bhattacharya's functional central limit theorem is equivalent
to the Kipnis-Varadhan central limit theorem for additive
functionals of time-reversible Markov processes.
\end{cor}
\begin{proof}
The above 
proof of Theorem \ref{kipvarthm} shows that Bhattacharya's
theorem implies the Kipnis-Varadhan theorem in the self-adjoint
case.  On the other hand,
the positive self-adjoint operator $-\hat{A}$
has a  unique positive square root $(-\hat{A})^{\frac{1}{2}}$
with 
the properties that $\mathcal{D}_{(-\hat{A})^{\frac{1}{2}}}
\supset\mathcal{D}_{-\hat{A}}$ and 
$\mathcal{D}_{(-\hat{A})^{-\frac{1}{2}}}\supset
\mathcal{D}_{(-\hat{A})^{-1}}={\mathcal R}_{\hat{A}}$.
In particular,  the Theorem \ref{kipvarthm} 
theorem implies
Theorem \ref{rabifcltthm} in the time-reversible case. 
\end{proof}
 Both  versions of a functional central limit theorem developed by \cite{RBfclt} and \cite{kipvar}
are notable for their applications to  solute dispersion in  \cite{guptabhat, skew, BhatWay3}, 
and to certain interacting particle systems in \cite{kipvar, liggett99, landim, sunder}, respectively. In addition to the 
order of limits problem, the approach here
may have added value in
computational and further theoretical refinements 
of the fclt, as well as to new directions of the
Gaussian random field theory
suggested by \cite{RBfclt, diaconisevans} for possble
applications in random matrix theory, quantum field theory,
or stochastic partial differential equations.  One immediate
refinement is that of the Strassen-type law of the iterated
logarithm established in \cite{RBfclt}. In particular, using 
the appropriately rescaled version of the
representation (\ref{intrep}),  one obtains the
 following equivalent
statement to that of the extension by
 \cite{RBfclt} in the time-reversible
case.
\begin{thm}
In the framework of Theorem \ref{rabifcltthm}
assume that $\hat{A}$ is a densely defined, self-adjoint infinitesimal
generator.
 If $f\in\mathcal{D}_{(-\hat{A})^{-\frac{1}{2}}}$ is centered, and
 if $\int_S|f|^{2+\delta}d\pi <\infty$ for some $\delta>0$, then
 $\{\frac{1}{\sqrt{2n\log\log n}}
\int_0^{nt}f(X(s))ds\}_{n\ge 2}$
 is almost surely
relatively compact in $C[0,1]$.  The set of limit
points is the set of all absolutely continuous functions
$\theta:[0,1]\to\mathbb{R}$ such that 
\begin{equation*}
\int_0^1\theta^\prime(x)^2dx \le \sigma^2(f), \quad\theta(0)=0.
\end{equation*}
\end{thm}
\begin{proof}
Choose decreasing $\lambda_n=o(\frac{1}{n\log\log n})$.
The two sequences $\{\frac{1}{\sqrt{2n\log\log n}}
\int_0^{nt}f(X(s))ds\}_{n\ge 2}$ and 
$\{\frac{1}{\sqrt{2n\log\log n}}
\int_0^{nt}(-\hat{A})R_{\lambda_n}f(X(s))ds\}_{n\ge 2}$ have
the same limit points since the $o(1)$ term in probability
has an almost surely $o(1)$ subsequence.
 Alternatively, one may simply
choose $\lambda_n\downarrow)$ sufficiently fast to obtain
an almost sure convergence in this representation. In addition,
  $\sigma^2(-\hat{A}R_{\lambda_n}f)$ increases
to $\sigma^2(f)$.
\end{proof}
\begin{remark} The condition
$\delta>0$ is removed in the recent arXiv 
paper \cite{chenchenhong}.
\end{remark}

  \section{Acknowledgments} The author thanks Sunder Sethuraman
for calling attention to \cites{kipvar,varadhan95,liggett99,landim}, and for sharing his unpublished
 notes \cite{sunder} on this fclt and its applications to certain
 particle systems. 
 The author is also grateful to Rabi Bhattacharya for 
 comments on the proof, and to 
 Persi Diaconis for encouraging remarks on an earlier draft.

\bibliography{fcltbib}
\end{document}